\documentclass{amsart}

\usepackage{graphics}
\usepackage{graphicx}
\usepackage{epsfig}
\usepackage{amscd}
\usepackage{amsrefs}
\usepackage{amssymb}
\usepackage{amsthm}
\usepackage{color}
\usepackage{extarrows}
\usepackage{amsmath}
\numberwithin{equation}{section}
\usepackage{mathrsfs}
\usepackage{hyperref}
\usepackage{bm}
\usepackage{float}
\usepackage{geometry}
\usepackage{tikz}
\usepackage{tikz-cd}
\usepackage{ytableau}
\usepackage{multirow}
\usepackage{verbatim}
\usepackage{stmaryrd}
\usepackage{geometry}
\usepackage[all,pdf]{xy}
\usepackage{lmodern,amssymb}

\usetikzlibrary{matrix}
\usetikzlibrary{arrows}
\usetikzlibrary{decorations.pathreplacing,decorations.markings}
\usetikzlibrary{matrix,arrows}
\usetikzlibrary{arrows,shapes,positioning}
\usetikzlibrary{decorations.markings}
\tikzstyle arrowstyle=[scale=1]
\tikzstyle directed=[postaction={decorate,decoration={markings,
    mark=at position .5 with {\arrow[arrowstyle]{stealth}}}}]
\tikzstyle reverse directed=[postaction={decorate,decoration={markings,
    mark=at position .65 with {\arrowreversed[arrowstyle]{stealth};}}}]

\newtheorem{thm}{Theorem}[section]

\newtheorem{cor}[thm]{Corollary}
\newtheorem{prop}[thm]{Proposition}
\newtheorem{conj}[thm]{Conjecture}

\theoremstyle{plain}
\newtheorem{lem}[thm]{Lemma}

\theoremstyle{remark}
\newtheorem{rem}[thm]{Remark}

\theoremstyle{definition}
\newtheorem{defn}[thm]{Definition}
\newtheorem{ex}[thm]{Example}

\numberwithin{equation}{section}

\newcommand{\A}{\mathscr{A}}

\newcommand{\C}{\mathbb{C}}
\newcommand{\F}{\mathbb{F}}
\newcommand{\codim}{\mathrm{codim}}

\newcommand{\Ho}{\mathfrak{h}}
\newcommand{\G}{\mathfrak{g}}
\newcommand{\Qo}{\mathfrak{q}}
\newcommand{\Lie}{\mathbb{L}}

\newcommand{\Q}{\mathbb{Q}}
\newcommand{\Z}{\mathbb{Z}}
\newcommand{\M}{\mathcal{M}}
\newcommand{\LL}{\mathcal{L}}
\newcommand{\T}{\mathcal{T}}
\newcommand{\Sal}{\mathcal{S}}
\newcommand{\bt}[1][normal]{\begin{tikzcd}[ampersand replacement = \&, column sep=#1,row sep=#1]}
\newcommand{\et}{\end{tikzcd}}

\begin{document}

\title{Holonomy Lie algebra of a geometric lattice}

\author{Weili Guo}
\address{Department of Mathematics, Beijing University of Chemical Technology, 15 Beisanhuan East Road, Chaoyang District, Beijing 100013, China}
\email{guowl@mail.buct.edu.cn}
\author{Ye Liu}
\address{Department of Pure Mathematics, Xi'an Jiaotong-Liverpool University, 111 Ren'ai Road, Suzhou, Jiangsu 215123 China}
\email{yeliumath@gmail.com}



\date{}



\begin{abstract}
Motivated by Kohno's result on the holonomy Lie algebra of a hyperplane arrangement, we define the holonomy Lie algebra of a finite geometric lattice in a combinatorial way. For a solvable pair of lattices, we show that the holonomy Lie algebra is an almost-direct product of the holonomy Lie algebra of the sublattice and a free Lie subalgebra. This yields the structure of the holonomy Lie algebra of a finite hypersolvable (including supersolvable) lattice. As applications, we obtain the structure of the holonomy Lie algebra of (the Salvetti complex of) a supersolvable oriented matroid, and that of a hypersolvable arrangement.
\end{abstract}

\maketitle

\tableofcontents

\section{Introduction}

For a hyperplane arrangement $\A$ in $\C^{\ell}$, the complement
\[
M=M(\A)=\C^{\ell}\setminus\bigcup_{H\in\A}H
\]
is of significant interest in the study of hyperplane arrangements. The fundamental group $G=\pi_1(M)$ is arguably the most important group associated to an arrangement. The lower central series of $G$ is 
\[
G=G_1\supset G_2\supset G_3\supset\cdots
\]
where $G_i=[G_{i-1},G_1]$ for $i\geq 2$. The graded vector space $\mathrm{gr}(G;\Q):=\left(\bigoplus G_i/G_{i+1}\right)\otimes\Q$ is in fact a Lie algebra, with bracket induced by group commutators, hence called the \emph{associated graded Lie algebra} of $G$. The dimension of each graded piece $\varphi_i(\A)=\varphi_i(G)=\dim \mathrm{gr}(G;\Q)_i=\mathrm{rank}(G_i/G_{i+1})$ is then an invariant of the arrangement $\A$. Evidently, $\varphi_1(\A)=\#\A$. In the literature, 
\[
U_{\A}(t):=\prod_{i=1}^{\infty}(1-t^i)^{\varphi_i(\A)}\in \Z[[t]]
\]
is used frequently to encapsulate all information about $\varphi_i(\A)$. An explicit equality relating $U_{\A}(t)$ with some polynomial is usually called a \emph{LCS (lower central series) formula}. The first such formula is  due to Kohno \cite{Kohno1985}, for braid arrangement: $$U_{\A}(t)=P_M(-t),$$ where $P_M(t)$ is the Poincar\'e polynomial of $M$. The same formula has been proved for \emph{fiber-type arrangements} by Falk and Randell \cite{Falk-Randell1985}, Shelton and Yuzvinsky \cite{Shelton-Yuzvinsky1997} also give an interpretation using Koszul duality. The LCS formula for $M$ a \emph{formal rational $K(\pi,1)$ space} has been given by Papadima and Yuzvinsky \cite{Papadima-Yuzvinsky1999}. An analogue formula for \emph{hypersolvable arrangements} has been proved by Jambu and Papadima \cite{Jambu-Papadima1998}. For \emph{decomposable arrangements}, Papadima and Suciu \cite{Papadima-Suciu2006} give a LCS formula. Lima-Filho and Schenck \cite{LimaFilho-Schenck2009} proves a LCS formula for graphic arrangements (subarrangements of braid arrangements), which was conjectured in \cite{Schenck-Suciu2002}.

The associated Lie algebra $\mathrm{gr}(G;\Q)$ is constructed in a group-theoretic fashion and thus not easy to deal with. Thanks to a theorem of Kohno \cite{Kohno1983}, (see Theorem \ref{Kohno}), it is naturally isomorphic to a combinatorially defined graded Lie algebra, the \emph{holonomy Lie algebra} $\Ho(M)$ (definition given in \S\ref{holo}), which is defined using the intersection lattice $L(\A)$ up to rank $2$. 

The study of this paper originates from generalizing this combinatorial definition of holonomy Lie algebra to general finite geometric lattices. This action has both advantages and disadvantages. The bright side is that once we enter the combinatorial framework, many existing lattice theoretic properties and arguments can show their full power in this larger scope, as well as new interesting combinatorial problems sprout. While the downside is also clear, that we lose geometry and topology due to lack of desirable geometric/topological models of geometric lattices (or matroids).

The structure of the paper is as follows. We recall basic definitions and properties of geometric lattice, solvable lattice pair and hypersolvable/supersolvable lattices in Section \ref{defn}. Orlik-Solomon algebra and holonomy Lie algebra of a geometric lattice are defined and studied in Section \ref{osvshlnm}. The main section of this paper is Section \ref{hlnmlatpair}, in which we investigate the structure of holonomy Lie algebras for solvable lattice pairs. Finally, we end with applications to supersolvable oriented matroids and hypersolvable arrangements in Section \ref{appl}.

\section{Definitions}\label{defn}
\subsection{Geometric lattice}
All lattices in what follows are assumed to be {\it finite}, unless otherwise stated.

A lattice $L$ is a poset in which the operations {\it join} and {\it meet} are defined for all pairs $x,y\in L$. Recall that the join $x\vee y$ and the meet $x\wedge y$ represent the least upper bound and the greatest lower bound for $x,y\in L$ respectively. Note that in a lattice $L$, join and meet can be defined for more elements, e.g. $\vee\{x,y,z\}=x\vee y\vee z$. A lattice $L$ has a unique minimal element $\hat{0}$, called the {\it bottom}, and a unique maximal element $\hat{1}$, called the {\it top}. For $x>y$ in $L$, we say $x$ {\it covers} $y$, written as $x\gtrdot y$ or $y\lessdot x$, if $x\geq z\geq y$ implies $z=x$ or $z=y$.

For a lattice $L$ and an element $x\in L$, we define a sublattice
\[
L_x:=\{y\in L\mid y\leq x\}.
\]

We denote by $A=A(L)$ the set of {\it atoms} of $L$, that are elements covering $\hat{0}$, and $A^{op}=A(L^{op})$ the set of {\it coatoms}, that are elements covered by $\hat{1}$, or equivalently, atoms of the opposite lattice $L^{op}$. We call a lattice $L$ {\it atomic} if every element is a join of atoms, where we agree that an atom is the join of itself and $\hat{0}$ is the join of $0$ atom.

A lattice $L$ is said to be {\it ranked} (or {\it graded}), if it admits a nonnegative integer valued function $r:L\to\Z_{\geq 0}$, such that $r(\hat{0})=0$ and $r(x)=r(y)+1$ whenever $x\gtrdot y$. We call such $r$ the {\it rank function} of $L$ and the {\it rank} of $L$ is defined by $r(L):=r(\hat{1})$. Dually, we call the rank function $r^{op}$ of $L^{op}$ as the {\it corank function} of $L$. Note that $r(x)+r^{op}(x)=r(L)$ for all $x\in L$. In a ranked lattice $L$, we denote
\[
L_k:=\{x\in L\mid r(x)=k\},
\]
and the truncation
\[
L_{\leq k}:=\{x\in L\mid r(x)\leq k\}.
\]
The atoms of $L$ are just rank $1$ elements $A=L_1$ and the coatoms of $L$ are just corank $1$ elements $A^{op}=L_{r(L)-1}$. A ranked lattice $L$ is {\it semimodular} if the following {\it semimodular inequality} holds for all $x,y\in L$,
\[
r(x)+r(y)\geq r(x\wedge y)+r(x\vee y).
\]

\begin{defn}
    A {\it geometric lattice} $L$ is a ranked lattice, which is both atomic and semimodular. An element of a geometric lattice is called a \emph{flat}.
\end{defn}

\begin{ex}
    A central hyperplane arrangement $\A$ is a finite set of (linear) hyperplanes in $V=\C^{\ell}$. The {\it intersection lattice} $L(\A)$ is the set of all nonempty intersections in $\A$, including the whole space $V=\C^{\ell}$ as the intersection over the empty set. The order of $L(\A)$ is defined by reverse inclusion, $x\leq y$ in $L(\A)$ if $x\supseteq y$. Then the top $\hat{1}=\bigcap\A=\bigcap_{H\in\A}H$ and the bottom $\hat{0}=V$. For $x,y\in L(\A)$, their meet is $x\wedge y=\bigcap\{z\in L(\A)\mid x\cup y\subseteq z\}$ and their join is $x\vee y=x\cap y$.
    
    The rank function of $L(\A)$ is given by
$$r(x)=\codim(x)=\ell-\dim(x).$$
It is well known that $L(\A)$ is a geometric lattice \cite{Orlik-Terao1992}.
\end{ex}

\begin{ex}[Example 3.1 of \cite{Falk1993}]\label{falk}
    Let $\A=\{H_1,\ldots,H_6\}$ be the hyperplane arrangement in $\C^3$ defined by $H_1=\{x+z=0\},H_2=\{x-z=0\},H_3=\{x+y=0\},H_4=\{y=0\},H_5=\{x-y=0\},H_6=\{z=0\}$. One checks that $H_{126}:=H_1\cap H_2\cap H_6=\{(0,y,0)\mid y\in\C\}$ and $H_{345}:=H_3\cap H_4\cap H_5=\{(0,0,z)\mid z\in\C\}$ have rank $2$. There are other $9$ rank $2$ flats formed by intersecting one of $H_1,H_2,H_6$ with one of $H_3,H_4,H_5$. See Figure \ref{falkex} for the Hasse diagram of $L(\A)$.
\begin{figure}
    \centering
       
\begin{tikzpicture}[scale=.7]
   \node (b) at (0,0) {$\hat{0}$};
   \node (1) at (-6,2) {$1$};
   \node (2) at (-4,2) {$2$};
   \node (6) at (-2,2) {$6$};
   \node (3) at (2,2) {$3$};
   \node (4) at (4,2) {$4$};
   \node (5) at (6,2) {$5$};
   \node (126) at (-10,5) {$126$};
   \node (13) at (-8,5) {$13$};
   \node (14) at (-6,5) {$14$};
   \node (15) at (-4,5) {$15$};
   \node (23) at (-2,5) {$23$};
   \node (24) at (0,5) {$24$};
   \node (25) at (2,5) {$25$};
   \node (36) at (4,5) {$36$};
   \node (46) at (6,5) {$46$};
   \node (56) at (8,5) {$56$};
   \node (345) at (10,5) {$345$};
   \node (t) at (0,7) {$\hat{1}$};
   \draw (b) -- (1) --(126) -- (t) -- (13) -- (1) -- (14) -- (t) -- (15) -- (1);
   \draw (b) -- (2) -- (23) -- (t) -- (24) -- (4) -- (b) -- (6) -- (126) -- (2)--(25)--(t)--(36)--(3)--(b)--(5)--(56)--(t)--(46)--(4)--(345)--(t);
   \draw (13)--(3)--(345)--(5)--(25);
   \draw (14)--(4);
   \draw (15)--(5);
   \draw (23)--(3);
   \draw (24)--(2);
   \draw (36)--(6)--(46);
   \draw (56)--(6);

\end{tikzpicture}
    \caption{Intersection lattice $L(\A)$}
    \label{falkex}
\end{figure}
 
\end{ex}

\begin{ex}
      A matroid is a generalization of independence relations among a set of vectors. There are various equivalent ways to define a matroid (see \cite{Oxley2011}). Let us define a matroid $\M=(E,r)$ as a finite ground set $E$ equipped with a map $r:2^E\to \Z_{\geq 0}$ satisfying
      \begin{enumerate}
          \item If $X\subseteq E$, then $0\leq r(X)\leq\#X$.
          \item If $X\subseteq Y\subseteq E$, then $r(X)\leq r(Y)$.
          \item If $X,Y\subseteq E$, then $r(X\cup Y)+r(X\cap Y)\leq r(X)+r(Y)$.
      \end{enumerate}
      This $r$ is called the \emph{rank function} of $\M$. The \emph{rank} of $\M$ is defined as $r(E)$. If $V$ is a vector space over a field $\F$ and $E$ is a finite set of vectors in $V$, define $r(X):=\dim_{\F}\mathrm{span}(X)$ for $X\subseteq E$. Then $(E,r)$ is matroid. A matroid isomorphic to such a $(E,r)$ is said to be $\F$-representable. Note that a complex central hyperplane arrangement $\A$ is simply a $\C$-representable matroid $(\A,r)$, where $r$ is the rank function of $L(\A)$.
      
      For a matroid $\M=(E,r)$, we define a map $cl:2^E\to 2^E$ by
      \[
      cl(X)=\{x\in E\mid r(X\cup\{x\})=r(X)\}.
      \]
      This map is called the \emph{closure operator} of $\M$ and $cl(X)$ is called the \emph{closure} of $X\subseteq E$. A closure is also called a flat. Define the set of flats
      \[
      L(\M):=\{cl(X)\mid X\subseteq E\}
      \]
      with partial order given by inclusion. It is a fact that $L(\M)$ is a lattice with join and meet defined by $X\vee Y=cl(X\cup Y)$ and $X\wedge Y=X\cap Y$ (Lemma 1.7.3 of \cite{Oxley2011}). The rank function $r$ of $\M$ also restricts to $L(\M)$. We call $L(\M)$ the \emph{flat lattice} of $\M$. Moreover, a lattice $L$ is geometric if and only if $L\cong L(\M)$ for some matroid $\M$ (Theorem 1.7.5 of \cite{Oxley2011}). In fact, there is a bijection between geometric lattices and \emph{simple} matroids (\cite{Oxley2011} p.54).
\end{ex}

\begin{ex} \label{fanoex}
    The Fano matroid $F_7$ is shown in Figure \ref{fano}.
    \begin{figure}
    \centering
        \begin{tikzpicture}

   \draw (0,0) circle (1);
   \draw (90:2) -- (-30:2)--(210:2)--cycle;

   \draw (90:2)--(0,0);
   \draw (210:2)--(0,0);
   \draw (-30:2)--(0,0);

   \draw (30:1)--(0,0);
   \draw (150:1)--(0,0);
   \draw (270:1)--(0,0);

   \fill (-1.732,-1) circle (1.5pt) node[left] {$a$};
   \fill (1.732,-1) circle (1.5pt) node[right] {$c$};
   \fill (0,-1) circle (1.5pt) node[below] {$b$};
   \fill (0,2) circle (1.5pt) node[right] {$e$};
   \fill (0,0) circle (1.5pt) node[above right] {$g$};
   \fill (0.866,0.5) circle (1.5pt) node[right] {$d$};
   \fill (-0.866,0.5) circle (1.5pt) node[left] {$f$};

\end{tikzpicture}
\caption{Fano matroid $F_7$}\label{fano}
    \end{figure}
    The nodes are atoms and three nodes on a line (or a circle) form a rank $2$ flat. The whole matroid $F_7$ has rank $3$. It is easy to depict the Hasse diagram of the flat lattice $L(F_7)$ (Figure \ref{fanolat}). 
    \begin{figure}
        \centering
        \begin{tikzpicture}[scale=.7]
   \node (0) at (0,0) {$\hat{0}$};
   \node (a) at (-6,2) {$a$};
   \node (b) at (-4,2) {$b$};
   \node (c) at (-2,2) {$c$};
   \node (d) at (0,2) {$d$};
   \node (e) at (2,2) {$e$};
   \node (f) at (4,2) {$f$};
   \node (g) at (6,2) {$g$};
   \node (abc) at (-6,5) {$abc$};
   \node (adg) at (-4,5) {$adg$};
   \node (aef) at (-2,5) {$aef$};
   \node (bdf) at (0,5) {$bdf$};
   \node (beg) at (2,5) {$beg$};
   \node (cde) at (4,5) {$cde$};
   \node (cfg) at (6,5) {$cfg$};
   \node (1) at (0,7) {$\hat{1}$};
   \draw (0) -- (a) -- (abc) -- (1) -- (adg) -- (d) -- (0) -- (b) -- (bdf) -- (1) -- (aef) -- (e) -- (0) -- (c) -- (cde) -- (1) -- (beg) -- (g) -- (0) -- (f) -- (cfg) -- (1);
   \draw (a) -- (adg) -- (g) -- (cfg) -- (c) -- (abc) -- (b) -- (beg) -- (e) -- (cde) -- (d) -- (bdf) -- (f) -- (aef) -- (a);
\end{tikzpicture}
        \caption{Flat lattice $L(F_7)$}
        \label{fanolat}
    \end{figure}
       It is known that $F_7$ is not $\C$-representable. In fact, $F_7$ is representable over a field $\F$ if and only if $\F$ has characteristic $2$ (\cite{Oxley2011} p.643).

\end{ex}

\subsection{Lattice pair}
We will study geometric lattices in pairs. That is, a geometric lattice and a sublattice generated by a subset of atoms satisfying some requirements. The following definitions and easy propositions are taken from Section 1 of \cite{Jambu-Papadima1998}, we point out that although Jambu and Papadima work in the representable case, that is $L=L(\A)$ for some central hyperplane arrangement $\A$, everything recorded here makes perfect sense for a general geometric lattice.

In what follows, let $L$ be a geometric lattice with atom set $A=L_1$.

\begin{defn}
    Let $B\subseteq A$ be a subset of atoms. We define the sublattice $L(B)$ of $L$ generated by $B$ as the lattice whose elements are joins of elements in $B$. Note that $L(B)$ inherits a geometric lattice structure and a rank function from $L=L(A)$ (while the corank function is shifted).
\end{defn}

Remark that $L(B)\subseteq L_{\vee B}$, the latter may have atoms outside of $B$.

\begin{defn}[Closed subsets]
	A proper subset $B\subsetneq A$ is said to be \emph{closed} in $A$ if $r(x\vee y\vee z)=3$ for any $x,y\in B~(x\neq y)$ and any $z\in A\setminus B$.
\end{defn}

\begin{defn}[Complete subsets]
	A proper subset $B\subsetneq A$ is \emph{complete} in $A$ if for any $x,y\in A\setminus B~(x\neq y)$, there exists $z\in B$ such that $r(x\vee y\vee z)=2$.
\end{defn}

\begin{lem}[Lemma 1.3 of \cite{Jambu-Papadima1998}]\label{cc}
	If $B$ is closed and complete in $A$, then
	\begin{enumerate}
	    \item $r(\vee A)-r(\vee B)\leq 1$.
        \item the element $z$ in the definition of completeness is uniquely determined by $x$ and $y$. Then we denote $z=f(x,y)=f(y,x)$.
	\end{enumerate}
\end{lem}

\begin{defn}[Solvable subsets]\label{solvable}
	A proper subset $B\subsetneq A$ is called a \emph{solvable subset} (or say the pair $B\subsetneq A$ is a \emph{solvable} extension) if $B$ satisfies the following three conditions
	\begin{enumerate}
	    \item $B$ is closed in $A$;
		\item $B$ is complete in $A$;
		\item for any distinct $x,y,z\in A\setminus B$, either $f(x,y)=f(y,z)=f(x,z)$ or $r(f(x,y)\vee f(y,z)\vee f(x,y))=2$.
	\end{enumerate}
\end{defn}

By Lemma \ref{cc} (1), we call a solvable extension $B\subsetneq A$ \emph{nonsingular} (resp. \emph{singular}) if $r(\vee A)-r(\vee B)=1$ (resp. $r(\vee A)=r(\vee B)$).

\begin{lem}[Lemma 1.7 of \cite{Jambu-Papadima1998}]\label{lem1.7}
    If $B\subsetneq A$ satisfies $r(\vee A)-r(\vee B)=1$, then
    \begin{enumerate}
          \item $B$ is solvable if and only if $B$ is complete in $A$.
        \item If $B$ is complete in $A$, then $B=A\cap L_{\vee B}$ (or equivalently $L(B)=L_{\vee B}$).
    \end{enumerate}
\end{lem}

\subsection{Hypersolvable vs supersolvable lattices}

\begin{defn}
We say $L$ is \emph{hypersolvable} if there is a filtration of atom subsets
\[
A_1\subsetneq A_2\subsetneq \cdots \subsetneq A_{\ell}=A
\]
such that $\#A_1=1$ and $A_i$ is solvable in $A_{i+1}$ for $i=1,2,\ldots,\ell-1$. The filtration is called a \emph{composition series} of $L$.
\end{defn}

\begin{ex}\label{falkcs}
    Let $L=L(\A)$ be the intersection lattice of the arrangement $\A$ in Example \ref{falk}. It is not difficult to verify
    \[
    \{H_1\}\subsetneq \{H_1,H_2,H_6\}\subsetneq \{H_1,H_2,H_6,H_3\}\subsetneq \A
    \]
    is a composition series of $L$. Hence $L$ is hypersolvable.
\end{ex}

The class of hypersolvable geometric lattices contains the class of supersolvable geometric lattices.

\begin{defn}
A pair $(x, y)\in L\times L$ is called a {\it modular pair} if for all $z\in L$ with $z\leq y$,
$$z\vee(x\wedge y)=(z\vee x)\wedge y.$$
An element $x\in L$ is called {\it modular} if $(x, y)$ is a modular pair for all $y\in L$.
\end{defn}
Note that a pair $(x, y)$ is modular if and only if $r(x)+r(y)=r(x\wedge y) + r(x\vee y)$ (Lemma 2.24 \cite{Orlik-Terao1992}).

\begin{lem}[Lemma 1.9 of \cite{Jambu-Papadima1998}]\label{modcork1}
    Suppose $x\in L$ has corank $1$. Then $x$ is modular if and only if $A_x:=A\cap L_x$ is solvable in $A$.
\end{lem}

\begin{defn}
A geometric lattice $L$ is called \emph{supersolvable} if $L$ has a maximal chain of modular flats
\[
\hat{0}=x_0\lessdot x_1\lessdot\cdots\lessdot x_r=\hat{1},
\]
where each $x_i$ is modular of rank $i$ and $r=r(L)$.
\end{defn}

\begin{ex}
    The flat lattice $L(F_7)$ of Fano matroid (Example \ref{fanoex}) is supersolvable since every flat is modular.
\end{ex}

\begin{prop}[Proposition 1.10 of \cite{Jambu-Papadima1998}]\label{sscs}
    If $L$ is supersolvable, then $L$ is hypersolvable with a composition series
    \[
A_1\subsetneq A_2\subsetneq \cdots \subsetneq A_{\ell}=A
\]
such that $r(\vee A_{i+1})=r(\vee A_i)+1$ for $i=1,2,\cdots,\ell-1$. In particular, $\ell=r(\vee A)=r(L)$.
\end{prop}

Nevertheless, Jambu and Papadima developed a deformation method \cites{Jambu-Papadima1998,Jambu-Papadima2002} to relate hypersolvable and supersolvable arrangements. We pose the following lattice theoretic counterpart as a conjecture.
\begin{conj}
    If $L$ is hypersolvable, then there exist $L^{\prime}$ which is supersolvable, and $L_{\leq 2}\cong L^{\prime}_{\leq 2}$.
\end{conj}
Jambu and Papadima \cites{Jambu-Papadima1998,Jambu-Papadima2002} proved this in the case $L$ is the intersection lattice of a central hyperplane arrangement. This is the part where they started to genuinely use $\C$-representability of the underlying matroid. In particular, their constructions and proofs are geometric and results from projective geometry, like Desargue's theorem, play an important role. We do not know whether the result holds if the underlying matroid is not $\C$-representable.

\section{Orlik-Solomon algebra vs holonomy Lie algebra}\label{osvshlnm}

\subsection{Orlik-Solomon algebra} In their study of cohomology ring of complement of a complex hyperplane arrangement, Orlik and Solomon \cite{Orlik-Solomon1980} find a pure combinatorial description in terms of the intersection lattice. Here we recall the definition of Orlik-Solomon algebra defined from a geometric lattice (for details, see \cite{Orlik-Terao1992}).

Let $L$ be a geometric lattice with atom set $A=L_1=\{x_1,x_2,\ldots,x_d\}$. Consider the (rational) exterior algebra $E^*=E^*(A)$ generated by degree $1$ generators $A$. We use notation like $xyz$ instead of $x\wedge y\wedge z$ to represent exterior products in $E^*$ to avoid confusions with meet. Then $E^*$ is graded and its $k$-th graded piece $E^k$ is a vector space over $\Q$ with basis $\{x_{i_1}\cdots x_{i_k}\mid 1\leq i_1<\cdots<i_k\leq d\}$. There is a boundary operator $\partial: E^k\to E^{k-1}$ defined as
\[
\partial(x_{i_1}\cdots x_{i_k})=\sum_{j=1}^k(-1)^{j-1}x_{i_1}\cdots \widehat{x_{i_j}}\cdots x_{i_k},
\]
where $\widehat{x_{i_j}}$ means deleting $x_{i_j}$. For an ordered subset $B=\{x_{i_1},\ldots,x_{i_k}\}(i_1<\cdots<i_k)$ of $A$, we denote $\Pi B:=x_{i_1}\cdots x_{i_k}\in E^k$.

\begin{defn}
Let $L$ be a geometric lattice with atom set $A$. Define the \emph{(rational) Orlik-Solomon algebra} of $L$ as
\[
OS^*(L):=E^*(A)/\langle \partial(\Pi B)\mid r(\vee B)<\#B\rangle.
\]
Note that the ideal in the quotient is graded and thus $OS^*(L)$ is graded.
\end{defn}

\begin{rem}\label{basis}
    The first graded piece $OS^1(L)$ is a vector space over $\Q$ with basis $\{[x_1],\ldots,[x_d]\}$, where $[x_i]$ is the class represented by $x_i\in E^1$. The second graded piece $OS^2(L)$ is a vector space over $\Q$ whose basis can be chosen as follows. For each $z\in L_2$, suppose the atoms covered by $z$ is $A_z=\{x_{i_1},x_{i_2},\ldots,x_{i_p}\}(i_1<i_2<\cdots<i_p)$, then pick $\{[x_{i_1}x_{i_2}],\ldots,[x_{i_1}x_{i_p}]\}$. The union for such sets when $z$ runs over $L_2$ forms a basis for $OS^2(L)$.
\end{rem}

The quadratic closure of $OS^*(L)$ is called the \emph{quadratic Orlik-Solomon algebra} $\overline{OS}^*(L)$, which is determined by $L_{\leq 2}$.

\begin{defn}
Let $L$ be a geometric lattice with atom set $A$. Define the \emph{(rational) quadratic Orlik-Solomon algebra} of $L$ as
\[
\overline{OS}^*(L):=E^*(A)/\langle \partial(\Pi B)\mid \#B=3,r(\vee B)=2\rangle.
\]
\end{defn}

Apparently from definitions, there is a canonical surjective homomorphism of graded algebras $\overline{OS}^*(L)\to OS^*(L)$.

\begin{prop}
    If $L$ is a supersolvable lattice, then the canonical homomorphism $\overline{OS}^*(L)\to OS^*(L)$ is an isomorphism.
\end{prop}
\begin{proof}
    The proof of Lemma 4.3 of \cite{Shelton-Yuzvinsky1997} is lattice-theoretic and hence works here.
\end{proof}

Jambu and Papadima proved the following theorem in the representable case. Again we point out that their proof in fact did not use representability and hence works for a general geometric lattice.

\begin{thm}[Theorem 2.3 of \cite{Jambu-Papadima1998}]\label{jpisom}
    $B$ is solvable in $A$ if and only if there is an isomorphism of graded $\overline{OS}^*(L(B))$-modules:
    \[
    \overline{OS}^*(L)\cong \overline{OS}^*(L(B))\otimes H^*(\bigvee_{\#(A\setminus B)}S^1;\Q),
    \]
    where the $\overline{OS}^*(L(B))$-module structure on $\overline{OS}^*(L)$ is induced by the inclusion of generators.
\end{thm}

Remark that the above isomorphism is a quadratic OS algebra analogue of the cohomology ring isomorphism of strictly linearly fibered arrangement studied in \cite{Falk-Randell1985}. 


Recall that the Hilbert-Poincar\'e series of a graded algebra $A^*=\bigoplus_{i\geq 0}A^i$ is $H_{A^*}(t)=\sum_{i\geq 0}\dim A^i$.

\begin{cor}[Proposition 3.2 of \cite{Jambu-Papadima1998}]
    Let $L$ be hypersolvable with a composition series $(\varnothing=A_0\subsetneq )A_1\subsetneq A_2\subsetneq \cdots \subsetneq A_{\ell}=A$. Then $H_{\overline{OS}^*(L)}(t)$ is a polynomial factored as
\[
H_{\overline{OS}^*(L)}(t)=\prod_{i=1}^{\ell}(1+d_it),
\]
where $d_i=\#(A_{i}\setminus A_{i-1})$ for $i=1,\ldots,\ell$. 
\end{cor}

\begin{proof}
    It follows inductively from the isomorphism in Theorem \ref{jpisom}.
\end{proof}

In particular, the degree of the polynomial $H_{\overline{OS}^*(L)}(t)$ for hypersolvable $L$ is $\ell$, which is independent of the choice of composition series. This number is called the length of $L$, denoted by $\ell=\ell(L)$. For a hypersolvable $L$, then there is an easy characterization of supersolvability using $\ell(L)$.

\begin{thm}[Theorem 3.4 of \cite{Jambu-Papadima1998}]
    For a hypersolvable $L$, it is supersolvable if and only if $\ell(L)=r(L)$.
\end{thm}
\begin{proof}[Sketch of proof]
    Suppose $L$ is supersolvable, then Proposition \ref{sscs} shows that $\ell(L)=r(L)$. Conversely, if $L$ is hypersolvable and $\ell(L)=r(L)$, choose a composition series $A_1\subsetneq A_2\subsetneq \cdots \subsetneq A_{\ell}=A$. Let $x_i:=\vee A_i$ for $i=1,\ldots,\ell$ and Lemma \ref{cc} (1) asserts that $r(x_i)=i$. Then an inverse induction on $i$ with the aid of Lemma \ref{lem1.7} and Lemma \ref{modcork1} shows that $\hat{0}=x_0\lessdot x_1\lessdot\cdots\lessdot x_{\ell}=\hat{1}$ is a maximal chain of modular flats. Thus $L$ is supersolvable. See Theorem 3.4 of \cite{Jambu-Papadima1998} for more details.
\end{proof}

\begin{ex}
    From the composition series in Example \ref{falkcs}, we know $\ell(L)=4$ and $r(L)=3$. So the hypersolvable $L$ in Example \ref{falkex} is not supersolvable.
\end{ex}

An algebraic result of hypersolvable $L$ is the Koszulness of $\overline{OS}^*(L)$, whose proof is the same as in Theorem 6.2 of \cite{Jambu-Papadima1998}. The latter is based on Theorem \ref{jpisom} and use induction on $\ell$, with an algebraic argument invoking a Hochschild-Serre type spectral sequence to prove the inductive step.
\begin{thm}[Theorem 6.2 of \cite{Jambu-Papadima1998}]\label{kos}
    For a hypersolvable $L$, the quadratic OS algebra $\overline{OS}^*(L)$ is Koszul.
\end{thm}

\subsection{Holonomy Lie algebra}\label{holo}
We fix some notations for free Lie algebras. For a finite set $S$, denote by $\Lie(S)$ the free Lie algebra on $S$. We say that $\Lie(S)$ has rank $\#S$ and sometimes simply denote by $\Lie(d)$ for a rank $d$ free Lie algebra. For a finite dimensional vector space $V$, let $\Lie(V)$ be the free algebra on a basis of $V$. The free Lie algebra $\Lie(S)$ is graded by depth of brackets $\Lie(S)=\bigoplus_{n}\Lie(S)_n$.

Let $M$ be a connected CW complex with finite $2$-skeleton. The dual of the map of cup product $H^1(M; \Q)\wedge H^1(M; \Q)\stackrel{\mu}{\longrightarrow}H^2(M; \Q)$ is the comultiplication map
$$H_2(M; \Q) \stackrel{\mu^{*}}{\longrightarrow} H_1(M; \Q)\wedge H_1(M; \Q)\hookrightarrow\Lie(H_1(M; \Q)),$$
where $\Lie(H_1(M; \Q))$ is the free Lie algebra on $H_1(M; \Q).$ Following Chen \cite{Chen1973}, we define the \emph{(rational) holonomy Lie algebra} $\Ho(M)$ of $M$ as the quotient
\[
\Ho(M):=\Lie(H_1(M; \Q))/\langle \mathrm{Im}~\mu^*\rangle,
\]
where $\langle \mathrm{Im}~\mu^*\rangle$ is the ideal generated by the image of the comultiplication $\mu^*$. Note that $\Ho(M)$ is determined by the cohomology ring structure $H^*(M;\Q)$ up to dimension $2$.

In \cite{Kohno1983}, Kohno gives a description of the holonomy Lie algebra for $M$ the complement of a hyperplane arrangement $\A$ in terms of the intersection lattice of $\A$ up to rank $2$. Motivated by his result, we introduce the holonomy Lie algebra of a geometric lattice.

\begin{defn}\label{hlnm}
    Let $L$ be a geometric lattice with set of atoms $A=L_1$. We define the {\it holonomy Lie algebra} $\Ho(L)$ of $L$ as
    \[
    \Ho(L):=\Lie(A)/I(L),
    \]
    where $I(L)$ is the ideal of $\Lie(A)$ generated by the set of generators
    \[
    \left\{\left[x,\Sigma z\right]\mid z\in L_2, x\in A, x<z \right\},
    \]
    where $\Sigma z=\sum_{y\in A,y<z}y\in\Lie(A)_1$.
\end{defn}
Note that $I(L)$ is a graded ideal $I(L)=\bigoplus_{n}I(L)_n$
, where $I(L)_n=I(L)\cap\Lie(A)_n$ and $I(L)_0=I(L)_1=0$. Then $\Ho(L)=\Lie(\A)/I(\A)$ has the natural grading. Note that $\Ho(L)$ is determined by flats up to rank $2$.

We may then state Kohno's result.

\begin{thm}[\cite{Kohno1983}]\label{Kohno}
Let $M$ be the complement of a complex hypersurface in $\C^{\ell}$, put $G=\pi_1(M)$, and $G=G_1\supset G_2\supset G_3\supset\cdots$ the lower central series of $G$. Then there is an isomorphism of graded Lie algebras
	\[
	\Ho(M)\cong \left(\bigoplus G_i/G_{i+1}\right)\otimes \Q.
	\]
	If in particular, $M$ is the complement of a hyperplane arrangement $\A$, we have
	\[
	\Ho(M)\cong\Ho(L(\A)).
	\]
\end{thm}

From this theorem, we can compute the invariant $\varphi_i(G)=\mathrm{rank}(G_i/G_{i+1})$ of arrangement $\A$ through calculating the dimension of $\Ho(L(\A))_i$.

\begin{rem}\label{isom}
    Using our choice of bases for $OS^1(L)$ and $OS^2(L)$ (Remark \ref{basis}), it is readily seen that our definition of $\Ho(L)$ is equivalent to
    \[
    \Ho(L)=\Lie(A)/\langle\mathrm{Im}~m^* \rangle,
    \]
    where $m^*$ is the dual of multiplication $m:OS^1(L)\wedge OS^1(L)\to OS^2(L)$ (or in the quadratic Orlik-Solomon algebra $m:\overline{OS}^1(L)\wedge \overline{OS}^1(L)\to\overline{OS}^2(L)$). In the case $L=L(\A)$ is the intersection lattice of an arrangement, this is exactly Kohno's result $\Ho(M)\cong \Ho(L(\A))$ since $H^*(M;\Q)\cong OS^*(L(\A))$ (\cite{Orlik-Solomon1980}). Moreover, if $S$ is a complex with $H^*(S;\Q)\cong OS^*(L)$ (or at least up to degree $2$) for some geometric lattice $L$, then $\Ho(S)\cong\Ho(L)$.
\end{rem}

By regarding $\overline{OS}^*(L)$ and the universal enveloping algebra $U(L)$ of $\Ho(L)$ as quotients of the tensor algebra $T^*(A)$ on atoms modulo respective relations, it is straightforward to make the following observation (cf. Lemma 4.1 of \cite{Shelton-Yuzvinsky1997}).
\begin{prop}\label{dual}
    For a geometric lattice $L$, the Koszul dual of its quadratic Orlik-Solomon algebra is isomorphic to the universal enveloping algebra of its holonomy Lie algebra
    \[
    \overline{OS}^*(L)^!\cong U(L).
    \]
\end{prop}
Combining Theorem \ref{kos} and Proposition \ref{dual}, we conclude the following corollary.
\begin{cor}
    For a hypersolvable $L$,
    \[
    H_{\overline{OS}^*(L)}(-t)=\prod_{i=1}^{\infty}(1-t^i)^{\dim_{\Q}\Ho(L)_i}.
    \]
\end{cor}
\begin{proof}
 Since $\overline{OS}^*(L)$ is Koszul, the Hilbert series satisfies $H_{U(L)}(t)H_{\overline{OS}^*(L)}(-t)=1$ (\cite{Priddy1970}). By Poincar\'e-Birkhoff-Witt's formula (Theorem 9.9 of \cite{Hall2015}), the right-hand-side of the equation to be proved is equal to $1/H_{U(L)}(t)$.
\end{proof}

\section{Holonomy Lie algebras of lattice pairs}\label{hlnmlatpair}
Let us begin our study of the holonomy Lie algebras of solvable lattice pairs.

\subsection{Closed subsets of atoms}

In this subsection, we shall consider a pair of atom sets $B\subsetneq A$, where $B$ is closed in $A$. The results here are generalizations of those in \cite{LimaFilho-Schenck2009}, where the authors consider graphic arrangement pairs which are \emph{triangle complete}. We point out that such arrangement pairs are also closed.
Firstly, we prove a technical lemma.

\begin{lem}\label{tech}
Assume $B\subsetneq A$ is closed.
		If $\alpha=[x_n,[x_{n-1},\ldots[x_1,x_0]\ldots]]\in\Lie(A)$, for some $x_i\in A\setminus B$, then in $\Ho(L)$,
		$$\alpha+I(L)=\alpha^F+I(L)$$
		for some $\alpha^F\in\Lie(A)$ where all entries are from $A\setminus B$.
	\end{lem}
	\begin{proof}
	This lemma is a generalization of Lemma 2.2 in \cite{LimaFilho-Schenck2009}. We prove by induction on $n$.

First suppose $n=1$, $\alpha=[x_1,x_0]$ and $x_1\in A\setminus B$, $x_0\in B$. If $z:=x_0\vee x_1\in L_2$ covers no other atoms $x^{\prime}\in A$, then
\[
[x_1,x_0]+I(L)=[x_1,x_0+x_1]+I(L)=0+I(L)=[x_1,x_1]+I(L),
\]
hence $\alpha^F=[x_1,x_1]$ is the desired representative. If $z=x_0\vee x_1\in L_2$ covers some other atoms, say all atoms covered by $z$ are $\{x_0, x_1, x_2,\ldots,x_s\}$. Since $B$ is closed, $x_1,x_2,\dots, x_s\in A\setminus B$, then we have
\[
[x_1,x_0]+I(L)=[x_1,x_0+x_1+\cdots+x_s]-[x_1,x_2]-\cdots-[x_1,x_s]+I(L)=\alpha^F+I(L),
\]
where $\alpha^F=[x_2,x_1]+\cdots+[x_s,x_1]$ is the desired representative.

Next suppose $\beta=[x_{n-1},\ldots,[x_1,x_0]\ldots]=[x_{n-1},\beta^{\prime}]$ has at least one entry from $A\setminus B$ and consider $\alpha=[x_n,\beta]$. By inductive hypothesis, it suffices to assume all entries of $\beta$ are from $A\setminus B$. Then the Jacobi identity gives
\[
[x_n,[x_{n-1},\beta^{\prime}]]=-[\beta^{\prime},[x_n,x_{n-1}]]-[x_{n-1},[\beta^{\prime},x_n]].
\]
The inductive hypothesis again proves the result. For the case $x_n\in A\setminus B$ and all entries of $\beta$ are from $B$, the result also follows from the above Jacobi identity and the inductive hypothesis.
	\end{proof}

For a geometric lattice $L=L(A)$ and a subset of atoms $B\subsetneq A$, note that the sublattice $L(B)$ has atom set $B$. Moreover, there is an inclusion $\sigma:\Lie(B)\to\Lie(A)$ between free Lie algebras, via which we may identify $\Lie(B)$ as a Lie subalgebra of $\Lie(A)$. In the other direction, the assignment
\[
\pi(x)=\begin{cases}
	x, & \text{ if } x\in B;\\
	0, & \text{ if } x\in A\setminus B.
\end{cases}
\]
extends to a surjective Lie algebra homomorphism $\pi:\Lie(A)\to\Lie(B)$. Note that $\pi\circ\sigma=id_{\Lie(B)}$. If $B\subsetneq A$ is closed, then each $z\in L_2$ either covers no atoms in $A\setminus B$ (equivalently $z\in L(B)$), or covers at most one atom in $B$. Hence for a generator $[x,\Sigma z]$ of $I(L)$, where $z\in L_2,x\in A,x<z$,
\[
\pi([x,\Sigma z])=\begin{cases}
	[x,\Sigma z], & \text{ if } z\in L(B);\\
	0, & \text{ otherwise }.
\end{cases}
\]

\begin{defn}
We call $[x,\Sigma z]$ above a \emph{vertical} (resp. \emph{horizontal}) generator of $I(L)$ with respect to $B$, if $z\in L(B)$ (resp. otherwise).   
\end{defn}

Clearly, generators of $I(L(B))$ are exactly the projections of vertical generators of $I(L)$, and we obtain $\pi(I(L))=I(L)\cap\Lie(B)=I(L(B))$. Therefore $\pi$ descends to a surjective Lie algebra homomorphism $\pi_*:\Ho(L)\to\Ho(L(B))$ with a section $\sigma_*:\Ho(L(B))\to\Ho(L)$ induced by $\sigma$. Note that $\sigma_*$ is injective since $I(L)\cap\Lie(B)=I(L(B))$. Thus $\Ho(L(B))$ is identified as the subalgebra of $\Ho(L)$ generated by $\{\alpha+I(L)\mid \alpha\in B\}$. In other words,
\[
\Ho(L(B))=\frac{\Lie(B)}{I(L(B))}=\frac{\Lie(B)}{\Lie(B)\cap I(L)}.
\]

To state the main result of this subsection, we introduce the following definition. It is the Lie algebra version of an \emph{almost-direct product of groups}, which was probably first studied by Falk-Randell \cite{Falk-Randell1985} and named by Jambu-Papadima \cite{Jambu-Papadima1998}.

\begin{defn}
	A Lie algebra $\G$ is an \emph{almost-direct product} of $\Qo$ and $\Ho$ if there is a split short exact sequence of Lie algebras
	\[
	0\to\Ho\to \G\overset{}{\underset{\sigma}\rightleftarrows}\Qo\to 0,
	\]
	and $[\Ho,\sigma(\Qo)]\subseteq[\Ho,\Ho]$. 
\end{defn}

\begin{prop}\label{LS}
	Assume $B\subsetneq A$ is closed. The kernel of $\pi_*$ is the subalgebra of $\Ho(L)$ generated by $\{\alpha+I(L)\mid \alpha\in A\setminus B\}$. In other words,
	\[
	\mathrm{Ker}~\pi_*\cong \frac{\Lie(A\setminus B)}{\Lie(A\setminus B)\cap I(L)}.
	\]
	Moreover, $\Ho(L)$ is an almost-directed product of $\Ho(L(B))$ and $\mathrm{Ker}~\pi_*$.
\end{prop}
\begin{proof}
	This proposition is a generalization of Corollary 2.3 in \cite{LimaFilho-Schenck2009}. In the following diagram of snake lemma,
		\begin{center}
\bt
  \&
  \& \mathrm{Ker} ~p_{A}| \ar{r}{} \ar{d}
  \& I(L)\ar{r}{\pi|} \ar{d}
  \& I(L(B)) \ar{d}   
  \&
  \&\\
    \&
    \& \mathrm{Ker}~\pi \ar{r}{} \ar{dd}[near start]{p_{A}|}
    \& \Lie(A) \ar{r}{\pi} \ar{dd}[near start]{p_{A}}
    \& \Lie(B)\ar{r}\ar{dd}[near start]{p_{B}}
    \&  0
    \& ~\\
        \&
        \&
        \& ~
        \&
        \&
          \ar[r, phantom, ""{coordinate, name=Y}]
        \&~\\
        ~\& \ar[l, phantom, ""{coordinate, name=Z}] 0 \ar{r}
        \& \mathrm{Ker}~\pi_* \ar{r}{} \ar{d}
        \& \Ho(L) \ar{r}{\pi_*} \ar{d}
        \& \Ho(L(B)) \ar{d}
        \&
            \& \\
              \&
              \& \ar[from=uuuurr, "\partial", crossing over, rounded corners,
                      to path=
                              { -- ([xshift=2ex]\tikztostart.east)
                              -| (Y) [near end]\tikztonodes
                              -| (Z) [near end]\tikztonodes
                              |- ([xshift=-2ex]\tikztotarget.west)
                               -- (\tikztotarget)}
                    ]\mathrm{Coker}~p_{A}|\ar{r}{}
               \& 0\ar{r}{}
               \& 0
               \&
               \&
    \et
  \end{center}
  since $\pi|_{I(L)}$ is onto, the connecting homomorphism $\partial$ is trivial. This proves that $p_{A}$ restricts to a surjection $\mathrm{Ker}~\pi\to\mathrm{Ker}~\pi_*$. Note that $\mathrm{Ker}~\pi$ is the \emph{ideal} of $\Lie(A)$ generated by elements in $A\setminus B$. Then $\mathrm{Ker}~\pi_*$ is the \emph{ideal} of $\Ho(L)$ generated by $\{\alpha+I(L)\mid \alpha\in A\setminus B\}$. By Lemma \ref{tech}, any element in $\mathrm{Ker}~\pi_*$ represented by a bracket containing entries from $A\setminus B$ can always be represented by sum of brackets containing no entries from $B$. To prove $\mathrm{Ker}~\pi_*$ is the \emph{subalgebra} of $\Ho(L)$ generated by $\{\alpha+I(L)\mid \alpha\in A\setminus B\}$, take an element of $\Ho(L)$ represented by $\beta\in\Lie(A)$ a sum of brackets whose entries are all from $B$, then $\pi_*(\beta+I(L))=0$ if and only if $\pi(\beta)=\beta\in I(L(B))$. Since $I(L(B))=I(L)\cap\Lie(B)$, we conclude that $\mathrm{Ker}~\pi_* $ does not contain nontrivial element represented by sum of brackets whose entries are all from $B$.
  
  The desired isomorphism is induced by the canonical quotient homomorphism $p_{A}:\Lie(A)\to\Ho(L)$.
  
  The inclusion $[\mathrm{Ker}~\pi_*,\Ho(L(B))]\subset[\mathrm{Ker}~\pi_*,\mathrm{Ker}~\pi_*]$ is a consequence of Lemma \ref{tech}.
\end{proof}

\subsection{Solvable subset of atoms}
We have seen that for a closed subset of atoms, the holonomy Lie algebra factors as an almost-direct product. In this subsection, we impose more requirements on a subset of atoms, which guarantees the kernel factor is a free Lie algebra.

We can prove the following key lemma.
\begin{lem}\label{J}
Let $L$ be a geometric lattice with atom set $A$. Suppose $B\subsetneq A$ is a solvable subset, we have
\[
\Lie(A\setminus B)\cap I(L)=0.
\]
\end{lem}
\begin{proof}
This is almost Part 2 of the proof of \cite{Jambu1990} Theorem 4.3.1, noting that the proof there is purely lattice theoretic and the conditions on $B\subsetneq A$ required there is exactly solvability. The only part of Jambu's proof that needs modification is the Step 1 there. We now reproduce the modified proof.

We shall prove the following claim. Suppose $a\in A\setminus B$ and $[b_1,b_2+\cdots+b_p]~(b_i\in B,p\geq 2)$ a vertical generator of $I(L)$, then $[a,[b_1,b_2+\cdots+b_p]]\in I(L)$. In fact,
\[
[a,[b_1,b_2+\cdots+b_p]]=\sum_{i=2}^p(-[[a,b_i],b_1]+[[a,b_1],b_i]).
\]
Consider $a\vee b_i\in L_2$, by closedness of $B$, we may suppose the atoms covered by $a\vee b_i$ are
\[
\{a,b_i,a_{ij}\mid j=1,\ldots,q(i)\},
\]
where $a_{ij}\in A\setminus B$. Then we may write
\[
-[[a,b_i],b_1]=-\left[\left[a,\sum_j a_{ij}+b_i\right],b_1 \right]+\sum_j[[a,a_{ij}],b_1],
\]
where $[a,\sum_j a_{ij}+b_i]$ is a horizontal generator of $I(L)$. Similarly,
\[
[[a,b_1],b_i]=\left[\left[a,\sum_k a_{1k}+b_1\right],b_i\right]-\sum_k[[a,a_{1k}],b_i].
\]
We need to further express the terms $[[a,a_{ij}],b_1]$ and $-[[a,a_{1k}],b_i]$ using Jacobi identity.
\begin{align*}
[[a,a_{ij}],b_1]&=[[a,b_1],a_{ij}]+[a,[a_{ij},b_1]]\\
&=\left[\left[a,\sum_ka_{1k}+b_1\right],a_{ij} \right]-\sum_k[[a,a_{1k}],a_{ij}]+\left[a,\left[a_{ij},\sum_m a_{1m}^{ij}+b_1\right]\right]-\sum_m[a,[a_{ij},a_{1m}^{ij}]],\\
-[[a,a_{1k}],b_i]&=-[[a,b_i],a_{1k}]-[a,[a_{1k},b_i]]\\
&=-\left[\left[a,\sum_ja_{ij}+b_i\right],a_{1k} \right]+\sum_j[[a,a_{ij}],a_{1k}]-\left[a,\left[a_{1k},\sum_l a_{il}^{1k}+b_i\right]\right]+\sum_l[a,[a_{1k},a_{il}^{1k}]].
\end{align*}
Now mod $I(L)$, we have
\begin{align*}
    [a,[b_1,b_2+\cdots+b_p]]&\equiv \sum_i\sum_j\sum_k[[a,a_{ij}],a_{1k}]+\sum_i\sum_j\sum_k[[a_{1k},a],a_{ij}]\\
    &+\sum_i\sum_j\sum_m[[a_{ij},a_{1m}^{ij}],a]+\sum_i\sum_k\sum_l[[a_{il}^{1k},a_{1k}],a].
\end{align*}
Note that the right hand side sum is in $\Lie(A\setminus B)$ and we need to prove it is $0$.

Note that $a,a_{1k},a_{ij}(i\neq 1)\in A\setminus B$ are distinct, by completeness and the third condition of solvability, $f(a,a_{1k})=b_1\neq f(a,a_{ij})=b_i$ implies that $f(a_{1k},a_{ij})<b_1\vee b_i$. Let $z:=a\vee a_{1k}\vee a_{ij}$, then we claim $r(z)=3$, since $a,b_1,b_i<z$ and $B$ is closed. We may thus also write $z=a\vee b_1\vee b_i$. Therefore $b_1,\ldots,b_p<z$ and consequently $a_{1k},a_{ij}<z$ for all $i,j,k$ since they are covered by $a\vee b_1$ or $a\vee b_i$.

Now we consider the atom $a_{1m}^{ij}$ which is covered by $a_{ij}\vee b_1$. Since $a,a_{ij},a_{1m}^{ij}\in A\setminus B$ are distinct by definition, and $f(a,a_{ij})=b_i\neq f(a_{ij},a_{1m}^{ij})=b_1$, we know $f(a,a_{1m}^{ij})$ is one of $b_2,\ldots,b_p$. For if $f(a,a_{1m}^{ij})=b_1$, the flat $x:=a\vee a_{1m}^{ij}\vee b_1\in L_2$ would also cover $a_{ij}$ and $b_i$, which already contradicts closedness. Say $f(a,a_{1m}^{ij})=b_s~(s\neq 1)$, then $a_{1m}^{ij}$ turns out to be $a_{s*}$. With the generator $[a_{ij},\sum_ma_{1m}^{ij}+b_1]$ also appear all other generators $[a_{1m}^{ij},\sum a_{1*}^{**}+b_1]$, which correspond to the same rank $2$ flat. Then $\sum_{i,j,m}[a_{ij},a_{1m}^{ij}]=0$. The third triple sum in question then vanishes.

Next consider $a\vee a_{il}^{1k}\in L_2$. It covers some $b_r(r\neq i)$. Hence $a_{il}^{1k}=a_{r*}$. Finally, consider $a_{1k}\vee a_{ij}\in L_2$. It must cover some $b_t$, then $a_{ij}=a_{t*}^{1k}$. Therefore the fourth triple sum $\sum_{i,k,l}[[a_{il}^{1k},a_{1k}],a]=\sum_{i,j,k}[[a_{ij},a_{1k}],a]$. Now the first, second and fourth triple sums in question cancel out by Jacobi identity. We finish proving the claim.

The remaining Steps 2-4 follows from Jambu's original proof without modifications needed.
 \end{proof}

Combining Proposition \ref{LS} and Lemma \ref{J}, we readily obtain the main result.
 \begin{thm}
     Let $L$ be a geometric lattice with atom set $A$. Suppose $B\subsetneq A$ is a solvable subset, then $\Ho(L)$ is an almost-direct product of $\Ho(L(B))$ and $\Lie(A\setminus B)$.
 \end{thm}

This result is the holonomy Lie algebra version of the analogous fundamental group result for solvable subarrangements of Jambu and Papadima (Theorem 4.3 of \cite{Jambu-Papadima1998}).

\begin{cor}\label{mainresult}
    If $L$ is a hypersolvable lattice with a composition series
\[
(\varnothing=A_0\subsetneq )A_1\subsetneq A_2\subsetneq \cdots \subsetneq A_{\ell}=A,
\]
then $\Ho(L)$ is an iterated almost-direct sum of the free Lie algebras $\Lie(d_i)$ where $d_i=\#(A_{i}\setminus A_{i-1})$ for $i=1,\ldots,\ell$. 
\end{cor}

\begin{ex}
    The holonomy Lie algebra $\Ho(L(F_7))$ of the Fano matroid is an iterated almost-direct sum of free Lie algebras $\Lie(1),\Lie(2)$ and $\Lie(4)$.
\end{ex}






\section{Applications}\label{appl}

\subsection{Oriented matroid}
An oriented matroid can be thought of as a generalization of a real hyperplane arrangement by abstracting the behaviors of sign vectors on the real hyperplane arrangement, which are meaningful thanks to the total order structure of real numbers. Questions about real hyperplane arrangements then can be asked also for oriented matroids.

Let $E$ be a finite set and a function $\sigma:E\to\{+,-,0\}$ is called a sign vector on $E$. The opposite of $\sigma$ is a sign vector $-\sigma$ defined by $(-\sigma)_e=-(\sigma_e)$ for $e\in E$. If $A\subseteq E$, the restriction of $\sigma$ to $A$ is then a sign vector $\sigma|_A$ on $A$. For two sign vectors $\sigma,\tau:E\to\{+,-,0\}$, their \emph{composition} (or \emph{Tits product}) is the sign vector $\sigma\circ\tau$ defined by
\[
(\sigma\circ\tau)_e=\begin{cases}
	\sigma_e, & \text{ if } \sigma_e\neq 0;\\
	\tau_e, & \text{ if } \sigma_e=0.
\end{cases}
\]
Their separating set $S(\sigma,\tau)$ is defined as
\[
S(\sigma,\tau)=\{e\in E\mid \sigma_e=-\tau_e\neq 0\}.
\]
\begin{defn}
    An oriented matroid is a pair $\M=(E,\LL)$, where $E$ is a finite set, called the ground set of $\M$, and $\LL\subseteq \{+,-,0\}^E$ is a set of sign vectors on $E$, called covectors of $\M$, such that
    \begin{itemize}
        \item $\mathbf{0}=(0,0,\ldots,0)\in \LL$,
        \item if $\sigma\in\LL$, then $-\sigma\in \LL$,
        \item if $\sigma,\tau\in\LL$, then $\sigma\circ\tau\in\LL$,
        \item if $\sigma,\tau\in\LL$ and $e\in E$, then there is a covector $\eta\in\LL$ such that $\eta_e=0$ and $\eta_f=(\sigma\circ\tau)_f=(\tau\circ\sigma)_f$ for all $f\in E\setminus S(\sigma,\tau)$.
    \end{itemize}
\end{defn}

Give an oriented matroid $\M=(E,\LL)$, we can define a poset structure on $\LL$ by saying $\sigma\leq\tau$ if $\sigma_e\leq\tau_e$ for all $e\in E$, where we regard $\{+,-,0\}$ as a poset with $0<+$ and $0<-$. The resulting poset $(\LL,\leq)$ is called the \emph{face poset} of $\M$, and the elements of $\LL$ are sometimes called faces. The maximal faces are called \emph{topes}. Write the set of topes by $\T=\T(\M)$. For a face $\sigma$, its \emph{zero set} (or \emph{flat}) $z(\sigma)$ is defined as $\{e\in E\mid \sigma_e=0\}$. A \emph{loop} of $\M$ is an element that is in the intersection of zero set of all faces. Two non-loop elements $e,f\in E$ are \emph{parallel} if $\sigma_e=0$ implies $\sigma_f=0$ for all $\sigma\in\LL$. We will consider only \emph{simple} oriented matroids, that is oriented matroids without loops or parallel elements.

\begin{defn}
    For an oriented matroid $\M=(E,\LL)$, let
    \[
    L(\M):=\{z(\sigma)\mid\sigma\in\LL\}.
    \]
    This is a poset with inclusion order. It is not difficult to see that $L(\M)$ is a geometric lattice.
\end{defn}

For a real hyperplane arrangement, Salvetti \cite{Salvetti1987} defined a regular cell complex which is homotopy equivalent to the complement of the complexified arrangement. This complex is later called the \emph{Salvetti complex} of the real arrangement and becomes the most important tool in studying the topology of real arrangements. The same construction in fact works for a general oriented matroid (\cite{Gelfand-Rybnikov1989}).

\begin{defn}
Let $\M$ be an oriented matroid and $\LL$ its face poset. The set of topes is $\T\subseteq\LL$. Define the \emph{Salvetti poset} $\Sal=\Sal(\M)$ by
\[
\Sal:=\{(\sigma,T)\in \LL\times\T\mid \sigma\leq_{\LL} T\},
\]
with partial order
\[
(\sigma,T)\leq_{\Sal}(\tau,R)\iff \sigma\geq_{\LL}\tau \text{ and }  \sigma\circ R=T.
\]
\end{defn}

We will view the Salvetti poset $\Sal$ as the topological counterpart of $\M$ by abusing notations with the geometric realization of its order complex, as we do in the case $\M$ is defined from a real hyperplane arrangement.

\begin{thm}[\cite{Gelfand-Rybnikov1989}]
    For an oriented matroid $\M$, there is an isomorphism of algebras
    \[
    H^*(\Sal(\M);\Q)\cong OS^*(L(\M)).
    \]
\end{thm}

From the above theorem and Remark \ref{isom}, we immediately obtain the following.
\begin{cor}
    For an oriented matroid $\M$, there is an isomorphism of Lie algebras
    \[
    \Ho(\Sal(\M))\cong \Ho(L(\M)).
    \]
\end{cor}

Then Corollary \ref{mainresult} implies the following.

\begin{cor}
Suppose an oriented matroid $\M=(E,\LL)$ is supersolvable, that is $L=L(\M)$ is supersolvable with a maximal modular chain,
\[
\hat{0}=x_0\lessdot x_1\lessdot\cdots\lessdot x_r=\hat{1}.
\]
Write $A_x$ as the set of atoms in $L_x$ and $d_i:=\#(A_{x_i}\setminus A_{x_{i-1}})$ for $i=1,\ldots,r$. Then the holonomy Lie algebra $\Ho(\Sal(\M))$ is an iterated almost-direct product of free Lie algebras $\Lie(d_1),\ldots,\Lie(d_r)$.
\end{cor}

The fundamental group version of this result is proved by Mücksch \cite{Mucksch2022}. See also Section 8 of \cite{Mucksch2022} for examples of non-representable supersolvable (oriented) matroids, to which our results also apply.

\begin{rem}
    For a matroid, it is not known whether an associated topological model has cohomology ring isomorphic to the Orlik-Solomon algebra. Therefore it is not known whether its holonomy Lie algebra is of topological meaning. There is a conjecture of Engstr\"om and Stamps (Conjecture 5.1.6 of \cite{Stamps2011}) asserting that such a model might be constructed from a codimension two homotopy sphere arrangement.
\end{rem}

\subsection{Hypersolvable arrangements}

Our main result may also apply to the $\C$-representable case, that is $L=L(\A)$ for a hypersolvable hyperplane arrangement $\A$. Together with Kohno's result Theorem \ref{Kohno}, we also obtain the rank $\varphi_i(\A)=\mathrm{rank}(G_i/G_{i+1})$, where $G_i$ is the $i$-th term in the lower central series of $\pi_1(M(\A))$.
\begin{thm}
	Let $\A$ be a hyperplane arrangement such that $L=L(\A)$ is hypersolvable with composition series 
 \[
(\varnothing=A_0\subsetneq )A_1\subsetneq A_2\subsetneq \cdots \subsetneq A_{\ell}=\A,
\]
then $\Ho(M(\A))$ is an iterated almost-direct sum of the free Lie algebras $\Lie(d_i)$ where $d_i=\#(A_{i}\setminus A_{i-1})$ for $i=1,\ldots,\ell$. In particular
\[
\varphi_j(\A)=\dim \Ho(M(\A))_j=\sum_{i=1}^{\ell}\dim \Lie(d_i)_j
\]
for all $j$.
\end{thm}

\begin{rem}
    For $\A$ a fiber-type arrangement, that is $L=L(\A)$ is supersolvable, Theorem 4.5 of \cite{Cohen+2003} gives a inductive description of $\Ho(M(\A))$, where the semidirect product structure is clearly written down.
\end{rem}

Consequently, we also recover the following LCS formula for hypersolvable arrangements.
\begin{cor}[\cites{Jambu-Papadima1998,Jambu-Papadima2002}]
	Let $\A$ be a hypersolvable arrangement as above, then
		\[
	\prod_{j=1}^{\infty}(1-t^j)^{\varphi_j(\A)}=\prod_{i=1}^{\ell}(1-d_it).
	\]

\end{cor}

\section*{Acknowledgement}
We would like to thank Michele Torielli and Masahiko Yoshinaga for valuable discussion. We are also grateful to Daniel C. Cohen for pointing us to the paper \cite{Cohen+2003}. Y. L. would like to thank Toshiyuki Akita for his invitation to a research stay in Hokkaido University, during which this work was initialed. W. G. is supported by NSFC grant No. 12201029. Y. L. is supported by NSFC grant No. 11901467.

\renewcommand\refname{Reference}
\bibliographystyle{hep}
\bibliography{ref}
\end{document}